\documentclass[10pt]{amsart}
\usepackage{amsmath,amssymb}
\usepackage{graphicx}

\usepackage{color}
\usepackage{combelow} 
\usepackage{hyperref} 
\usepackage[all]{xy} 
\usepackage{indentfirst} 

\usepackage[a4paper]{geometry}
\newtheorem{theorem}{Theorem}[section]
\newtheorem{lemma}[theorem]{Lemma}
\newtheorem{corollary}[theorem]{Corollary}
\newtheorem{proposition}[theorem]{Proposition}

\newtheorem{theorem-B}{Theorem}
\newtheorem{lemma-B}[theorem-B]{Lemma}
\newtheorem{corollary-B}[theorem-B]{Corollary}
\newtheorem{proposition-B}[theorem-B]{Proposition}

\theoremstyle{definition}
\newtheorem{remark}[theorem]{Remark}

\newtheorem{subsec}[theorem]{}

\numberwithin{equation}{section}

\DeclareMathAlphabet\mathbfcal{OMS}{cmsy}{b}{n}

\DeclareFontFamily{OT1}{pzc}{}
\DeclareFontShape{OT1}{pzc}{m}{it}{<->s*[1.10] pzcmi7t}{}
\DeclareMathAlphabet{\mathscr}{OT1}{pzc}{m}{it}

\begin{document}

\hfuzz=13pt
\title[Support $\tau$-tilting modules and semibricks]{Support $\tau$-tilting modules and semibricks \\ over  group graded algebras}
\author[S. Breaz]{Simion Breaz}
\author[A. Marcus]{Andrei Marcus}
\author[G. C. Modoi]{George Ciprian Modoi}
\address{\parbox{\linewidth}{Babe\cb{s}-Bolyai University, Faculty of Mathematics and Computer Science, Department of Mathematics, 1 Kog\u{a}lniceanu St., 400084, Cluj-Napoca, Romania}}
\email{bodo@math.ubbcluj.ro}
\email{marcus@math.ubbcluj.ro}
\email{cmodoi@math.ubbcluj.ro}


\begin{abstract} We consider a finite dimensional strongly $G$-graded algebra $A$ with { self-injective} $1$-component $B$, and in our main result we prove that the induction from $B$ to $A$ of a basic support $\tau$-tilting pair of $B$-modules is a support $\tau$-tilting pair $(M,P)$ of $A$-modules if and only if $M$ is $G$-invariant. A similar statement holds for the restriction from $A$ to $B$, so our results may be viewed as  Clifford and Maschke type theorems for $2$-term silting complexes. We also give applications to semibricks and the associated wide subcategories.
\end{abstract}

\maketitle

\section{Introduction}

{ This article is motivated by the work done by R.~Koshio and Y.~Kozakai in \cite{art:KK2020} and \cite{art:KK2021}. They consider a block algebra  $B$ of a normal subgroup $G$ of a finite group $\tilde G$, and under some strong assumptions, they obtain correspondences between the $2$-term tilting complexes over $B$, and the $2$-term tilting complexes over $\tilde B$, where $\tilde B$ is a block algebra of $\tilde G$ covering $B$. Similar correspondences are obtained between sets of isomorphism classes of associated objects: support $\tau$-tilting modules, semibricks, simple minded collections. Their results rely on the study of the behaviour of these objects under the functor $\tilde B\cdot \mathrm{Ind}_G^{\tilde G}$. }

{ We show here that the assumptions of \cite{art:KK2020} and \cite{art:KK2021} can be significantly relaxed. We build upon the ideas and results of \cite{art:B2020}  on induction of silting complexes, and of  \cite{art:M2003},  \cite{art:M2005} and  \cite{art:MPan2019}, where group graded tilting complexes  are constructed from Rickard's simple minded collections.}

{We consider a finite dimensional strongly $G$-graded algebra $A$ with $1$-component $B$ over an arbitrary field $k$. In this situation, the induction functor $\mathrm{Ind}^A_B=A\otimes_B-$ and the restriction of scalars $\mathrm{Res}^A_B$ from $A$ to $B$ form a Frobenius (biadjoint) pair. Note that $\mathrm{Ind}^A_B$ is always separable, while  $\mathrm{Res}^A_B$ is separable when the characteristic of $k$ does not divide the order $|G|$ of $G$. The crucial observation made in Theorem \ref{t:gr-cover} is that { if $B$ is self-injective, then} $\mathrm{Ind}^A_B$ commutes with projective covers, and hence takes a minimal projective presentation to a minimal projective presentation. E.~C.~Dade \cite{art:Dade1984}  proved a similar result for $\mathrm{Res}^A_B$  when  $\mathrm{char}\,k \nmid |G|$, and we also simplify his proof in that case, avoiding the use of Clifford extensions.}

{This has a number of consequences which allow us to show in Theorem  \ref{t:main-ind-2-silt} that if $M$ is a support $\tau$-tilting $B$-module, then $\mathrm{Ind}^A_BM$ is a support $\tau$-tilting $A$-module if and only if $M$ is $G$-invariant. A similar result involving $\mathrm{Res}^A_B$ also holds, and we point out that there is a general phenomenon which can be expressed in terms of Frobenius and separable functors that will be discussed in a subsequent paper. Nevertheless, several of our results stated for strongly group graded algebras do not hold for much more general Frobenius ring extensions.}

{The outline of the paper is as follows.  In Section  \ref{s:tau-tilt-2-silt} we present the needed results on $2$-term silting complexes, support $\tau$-tilting modules and semibricks. 
}

{Section  \ref{sec:graded} contains results on projective covers and Auslander-Reiten translations. Our main result here says that for any $B$-module $Y$, we have $\mathrm{Ind}^A_B(\tau_B (Y)) \simeq \tau_A(\mathrm{Ind}^A_B Y)$, this being an important ingredient in the proof of Theorem  \ref{t:main-ind-2-silt}. Moreover, if $\mathrm{char}\,k \nmid |G|$ then  $\mathrm{Res}^A_B(\tau_A (X)) \simeq \tau_B(\mathrm{Res}^A_B X)$  for any $A$-module $X$. We also present in this section the facts from Clifford theory which will be used in the subsequent sections.  In Section \ref{sec:ind-res-tau} we prove Theorem   \ref{t:main-ind-2-silt} and also its counterpart for $\mathrm{Res}^A_B$. }

{In Section \ref{s:semibrick}  we study left finite semibricks and functorially finite torsion classes over $B$ respectively $A$. Theorem \ref{f:bricks}, which  generalizes results of Koshio and Kozakai, is an application of the so-called stable Clifford correspondence and of the results of the preceding sections. We explain in detail the relationship with the results from \cite{art:KK2020}, \cite{art:K2022} and \cite{art:KK2021}. }

In Section \ref{sec:wide} we apply our framework to Asai's results from \cite[Sections 2.3 and 3.3]{art:Asai2020} on the realization of certain wide subcategories as module categories. We show here that we get strongly $G$-graded endomorphism algebras, and that the category equivalences given in \cite[Theorems 2.27 and 3.15]{art:Asai2020} commute with induction and restriction.

Note that we have at our disposal the induction and restriction functors, and also silting and cosilting complexes. Therefore, many of our statements may have four variants, but our main concern here is the induction of  $2$-silting complexes.

{All rings in this paper are associative with identity and all modules are left (unless otherwise specified) unital and finitely generated.  Our general assumptions and conventions are standard, and  tend to follow \cite{art:AIR2014} and \cite{art:Asai2020}. The field $k$ is arbitrary, and $k$-algebras are finite dimensional. For a module $X$ over the $k$-algebra $A$, $X^*$ is the $A$-dual of $X$, $\mathrm{D}X$ is the $k$-dual, while $\mathrm{P}$, $\Omega$, $\nu$, $\mathrm{Tr}$, $\tau$ denote the projective cover, the syzygy functor, the Nakayama functor, the transpose, and the Auslander-Reiten translation, respectively. For a full subcategory $\mathcal{C}$ of $\mathrm{mod}\,A$, the notations $\mathrm{Fac}\,\mathcal{C}$, $\mathrm{Sub}\,\mathcal{C}$, $\mathrm{Filt}\,\mathcal{C}$, $\mathrm{T}(\mathcal{C})$ are also standard. For basic concepts and results on representation theory, Frobenius and separable functors, and group graded algebras, we refer the reader to \cite{book:ARS}, \cite{book:ASS}, and \cite{book:M1999}.}

\section{Preliminaries on support $\tau$-tilting modules} \label{s:tau-tilt-2-silt}

As we have mentioned in the Introduction, our main objects of study are the $2$-silting complexes and their corresponding support $\tau$-tilting modules and left finite semibricks, and also the dual versions of them. We recall from \cite{art:AIR2014} and \cite{art:Asai2020} the main facts which will be used we here.

\begin{subsec} \label{ss:notations} Let $B$ be a finite dimensional $k$-algebra. By \cite[Theorem 3.2]{art:AIR2014}, there is a bijection between the set $2$-$\mathrm{silt}\,B$ of isomorphism classes  of 2-term silting complexes of projective $B$-modules  and the set $\mathrm{s}\tau$-$\mathrm{tilt}\,B$ of isomorphism classes  of support $\tau$-tilting $B$-modules.

Let $\mathbf{P}\in \mathcal{K}^{\mathrm{b}}(\mathrm{proj}\,B)$ be in $2\textrm{-silt}\,B$, and let $M=\mathrm{H}^0(\mathbf{P})$  be the support $\tau$-tilting module corresponding to $\mathbf{P}$.  There is  an idempotent $e \in B$ such that $M$ is a $\tau$-tilting $B/\langle e\rangle$-module, where $\langle e\rangle=BeB$, and the composition factors of $M$ are exactly the simple submodules of $(1-e)B/(1-e)\mathrm{rad}B$.

We will assume that $M$ is basic and that the complex
\[\mathbf{P}=(P_1\oplus P\overset{[f\, 0]}\longrightarrow P_0)\] is obtained via  a minimal projective presentation
\[P_1\overset{f}\longrightarrow P_0\to M\to 0\] of $M$. Recall that here $P=Be$, and $(M,P)$ is a support $\tau$-tilting pair for $B$.

Denote \[D=\mathrm{End}_{\mathcal{K}^{\mathrm{b}}(\mathrm{proj}\,B)}(\mathbf{P})^{\textrm{op}}\qquad \textrm{ and }\qquad F=\mathrm{End}_B(M)^{\textrm{op}}.\]

As in \cite[Theorem 2.3]{art:Asai2020}, let \[S=M/L\in \mathrm{f}_\mathrm{L}\textrm{-sbrick}\,B\] be the left-finite semibrick corresponding bijectively to $M$, where \[L=R(M,M)=\mathrm{rad}_FM.\] Note that if $M$ is basic, then $S$ is also basic.
\end{subsec}

\begin{subsec}  Consider the Nakayama functor \[\nu:\mathcal{K}^{\mathrm{b}}(\mathrm{proj}\,B)\to \mathcal{K}^{\mathrm{b}}(\mathrm{inj}\,B)\] and let  $\mathbf{P}':=\nu\mathbf{P}$ Then $\nu\mathbf{P}\in 2\textrm{-}\mathrm{cosilt}\,B$.

Let $M'=H^{-1}(\nu\mathbf{P})\in \mathrm{s}\tau^{-1}\textrm{-}\mathrm{tilt}\,B$ and $S'=\mathrm{soc}_B(M)\in \mathrm{f}_\mathrm{R}\textrm{-}\mathrm{sbrick}\,B$. Denote also \[D':=\mathrm{End}_{\mathcal{K}^{\mathrm{b}}(\mathrm{inj}\,B)}(\nu\mathbf{P})^{\textrm{op}}\qquad \textrm{ and }\qquad F'=\mathrm{End}_B(M')^{\textrm{op}}.\]
Then $\nu$ induces an isomorphism between $D$ and $D'$.
\end{subsec}

\begin{subsec} We consider the decomposition $\mathbf{P}=\bigoplus_{i=1}^nP_i$, where $P_i\in \mathcal{K}^{\mathrm{b}}(\mathrm{proj}\,B)$ are indecomposable. We know that $n=|B|$. Then we have the decomposition $\mathbf{P}'=\bigoplus_{i=1}^nP'_i$, where $P'_i=\nu P_i\in \mathcal{K}^{\mathrm{b}}(\mathrm{inj}\,B)$ are indecomposable.

Let $M_i=H^0(P_i)$, hence $M=\bigoplus_{i=1}^nM_i$. Let $L_i=\sum_{f\in\mathrm{rad}_A(M,M_i)}\mathrm{Im}f$ and $S_i=M_i/L_i$, for $i=1,\dots,n$. We have $L=\bigoplus_{i=1}^nL_i$ and $S=\bigoplus_{i=1}^nS_i$.

Let $M'_i:=H^{-1}(\nu P_i)$ and $S'_i=\bigcap_{f\in\mathrm{rad}_A(M'_i,M')}\mathrm{Ker}f$. We have $M'=\bigoplus_{i=1}^nM'_i$ and $S'=\bigoplus_{i=1}^nS'_i$.
\end{subsec}

\begin{subsec} \label{ss:indices} By the results of \cite[Section 3.3]{art:Asai2020}, we may assume that there are natural numbers $1\le k\le l\le m\le n$, and that we may chose the indices such that the following properties hold.

\begin{enumerate}
\item  For $1\le i\le m$ we have that $M_i\neq 0$ is an indecomposable $\tau$-rigid $B$-module, while for $k+1\le i\le n$ we have that $M'_i\neq 0$ is an indecomposable $\tau^{-1}$-rigid $B$-module.
\item $M_i \notin \mathrm{Fac}(M/M_i)$ if and only if $1\le i\le l$. In this case, $S_i\simeq M_i/R(M_i,M_i)\in \mathrm{f}_\mathrm{L}\textrm{-}\mathrm{brick}\,B$, and $S'_i=0$.
\item For $l+1\le i\le n$ we have that $S'_i\in \mathrm{f}_\mathrm{R}\textrm{-}\mathrm{brick}\,B$, and $S_i=0$.
\item For $1\le i\le k$ we have that  $M_i$ is projective and $M'_i=0$, while for $m+1\le i\le n$ we have that $M'_i$ is injective and $M_i=0$.
\end{enumerate}
\end{subsec}

\begin{subsec} \label{s:bijections} In addition to the above mentioned bijections between the sets $2$-$\mathrm{silt}\,B$, $\mathrm{s}\tau$-$\mathrm{tilt}\,B$ and $\mathrm{f}_\mathrm{L}\textrm{-}\mathrm{sbrick}\,B$, by \cite[Theorem 2.7]{art:AIR2014} and \cite[Proposition 2.24 and Proposition 3.5]{art:Asai2020} these sets are also in bijection with:
\begin{enumerate}
\item The set $\mathrm{f}\textrm{-}\mathrm{tors}\,B$ of functorially finite torsion classes $\mathcal{T}$ in $\mathrm{mod}\,B$.
\item The set $\mathrm{f}_\mathrm{L}\textrm{-}\mathrm{wide}\,B$  of left finite wide subcategories $\mathcal{W}$ of $\mathrm{mod}\,B$.
\item The set $\mathrm{int}\textrm{-}\mathrm{t}\textrm{-}\mathrm{str}\,B$ of intermediate t-structures with length heart $\mathcal{H}$ in $\mathcal{D}^{\mathrm{b}}(\mathrm{mod}\,B)$.
\end{enumerate}
Denote by $\mathcal{T}$ the functorially finite torsion class,  by $\mathcal{W}$ left finite wide subcategory, and by $\mathcal{H}$ the heart of the intermediate t-structure corresponding to $M$. We know that $\mathcal{T}=\mathrm{Fac}\,M=\mathrm{T}(S)=\mathrm{Filt}(\mathrm{Fac}\,S)$, and  $\mathcal{W}=\mathrm{Filt}\,S$. Under this bijection, $S$ becomes a semisimple object in $\mathcal{W}$ and also in the abelian category $\mathcal{H}$.
\end{subsec}

\begin{subsec}\label{s:stable}
{ In the following we will prove a slight generalization of \cite[Theorem 2.2]{art:B2020}, also adapted to the case of $\tau$-tilting modules, that will be used in the proofs of the main results of this paper.

{In order to do it, let us recall that if $\mathbf{P}: P_{1}\overset{f}\to P_0$ is a complex of projective $B$-modules that is concentrated in the degrees $0$ and $1$, we can associate to it, as in \cite{art:AMV2015}, the class $\mathcal{D}_{\mathbf{P}}=\{K\in B\textrm{-}\mathrm{mod}\mid \mathrm{Hom}_{\mathbf{D}^b(A)}(\mathbf{P}, K[1])=0\}$. We know from \cite{art:AMV2015} that $\mathbf{P}$ is silting if and only if $\mathcal{D}_{\mathbf{P}}=\mathrm{Fac}(H^0(\mathbf{P}))$.
}

Suppose that $\mathcal{F}:B\textrm{-}\mathrm{mod}\to A\textrm{-}\mathrm{mod}:\mathcal{G}$ is a pair of functors such that $\mathcal{F}$ is a left adjoint of $\mathcal{G}$. We say that a module $M$ is \textit{$\mathcal{F}$-stable} if  $\mathcal{GF}(M)\in\mathrm{add}\,M.$


\begin{proposition}\label{prop:var-ext-2-silt}
Suppose that $A$ and $B$ are $k$-algebras, and that $\mathcal{F}:B\textrm{-}\mathrm{mod}\to A\textrm{-}\mathrm{mod}:\mathcal{G}$ are functors such that

\noindent {\rm(a)} $\mathcal{F}$ is a left adjoint for $\mathcal{G}$,

\noindent {\rm(b)} for every $K\in A\textrm{-}\mathrm{mod}$ the canonical morphism $\mathcal{FG}(K)\to K$ is surjective, and

\noindent {\rm(c)} $\mathcal{F}$ preserves the projective property (or, equivalently, $\mathcal{G}$ is exact).

\noindent If $\mathbf{P}$ is a $2$-silting complex with $H^0(\mathbf{P})=M$, the following are equivalent:

{\rm(1)}  $\mathcal{F}(\mathbf{P})\in 2\textrm{-}\mathrm{silt}\,A$;

{\rm(2)} $\mathcal{GF}(M)\in\mathrm{Fac}\,M.$

\noindent Moreover, if

\noindent {\rm(d)} $\tau \mathcal{GF}(M)\cong \mathcal{G}(\tau \mathcal{F}( M))$,
\noindent the above conditions are equivalent to

{\rm(3)} $M$ is $\mathcal{F}$-stable.
\end{proposition}

\begin{proof}
Since $\mathcal{F}\dashv \mathcal{G}$, it is easy to see $H^0(\mathcal{F}(\mathbf{P}))=\mathcal{F}(M)$ and that $K\in \mathcal{D}_{\mathcal{F}(\mathbf{P})}$ if and only if $\mathcal{G}(K)\in \mathcal{D}_{\mathbf{P}}$ (see the proof of \cite[Lemma 2.1]{art:B2020}).

(1)$\Rightarrow$(2) Assume that $\mathcal{F}(\mathbf{P})$ is a silting complex. Then $\mathcal{F}(M)\in \mathcal{D}_{\mathcal{F}(\mathbf{P})}$, hence $\mathcal{GF}(M)\in \mathcal{D}_{\mathbf{P}}=\mathrm{Fac}\,M$.

(2)$\Rightarrow$(1) Suppose that $\mathcal{GF}(M)\in\mathrm{Fac}\,M.$ For every $K\in \mathcal{D}_{L(\mathbf{P})}$ we have $\mathcal{G}(K)\in \mathcal{D}_{\mathbf{P}}$, hence there is an epimorphism $M^n\to \mathcal{G}(K)$. Since $\mathcal{F}$ is right exact, we obtain an epimorphism $\mathcal{F}(M)^n\to \mathcal{FG}(K)$, and using the hypothesis (b) we conclude that $K\in \mathrm{Fac}\,\mathcal{F}(M)$. Then $\mathcal{D}_{\mathcal{F}(\mathbf{P})}\subseteq \mathrm{Fac}\,\mathcal{F}(M)$. The inclusion $\mathrm{Fac}\,\mathcal{F}(M)\subseteq \mathcal{D}_{\mathcal{F}(\mathbf{P})}$ follows from the following observations: by (2) and the equivalence mentioned in the beginning of the proof we have $\mathcal{F}(M)\in \mathcal{D}_{\mathcal{F}(\mathbf{P})}$, and the class $\mathcal{D}_{\mathcal{F}(\mathbf{P})}$ is closed with respect to epimorphic images.

(2)$\Rightarrow$(3) We associate to $\mathbf{P}$ the support $\tau$-tilting pair $(P,M)$. We know that  $\mathcal{GF}(M)\in \mathcal{D}_{\mathbf{P}}$,  hence \[\mathrm{Hom}_B(\mathcal{GF}(M), \tau M)=0 \]
and
\[ \mathrm{Hom}_B(P,\mathcal{GF}(M))=0. \]
Moreover, by using our hypotheses, we obtain
\begin{align*}
	\mathrm{Hom}_B(M, \tau \mathcal{GF}(M)) &\simeq \mathrm{Hom}_B(M, \mathcal{G}(\tau \mathcal{F}( M))) \\
	&  \simeq \mathrm{Hom}_A(\mathcal{F}(M), \tau \mathcal{F}( M)) =0.
\end{align*}
By \cite[Corollary 2.13]{art:AIR2014} we deduce that $\mathcal{GF}(M)\in \mathrm{add}\,M$.
\end{proof}
}
\end{subsec}


\section{Modules over group graded algebras and the Auslander-Reiten translation} \label{sec:graded}

\begin{subsec} Let $A=\bigoplus_{g\in G}A_g$ be a finite dimensional strongly $G$-graded $k$-algebra with 1-component $A_1=B$, so $\phi:B\to A$ is the inclusion. In this particular case,  $(\mathrm{Ind}^A_B,\mathrm{Res}^A_B)$ is a Frobenius pair, and $\mathrm{Ind}^A_B$ is separable. Moreover, if the characteristic of $k$ does not divide the order of $G$, then $\mathrm{Res}^A_B$ is also separable. A more general result on the separability of $\phi$ is given in \cite[Proposition 2.1]{art:NVV1989}. Note that these conditions have been considered in \cite{art:RR1985} and \cite[Section III. 4]{book:ARS}.
\end{subsec}

\begin{subsec} \label{s:g-conj} We will also rely on other properties of the ring extension $\phi:B\to A$ in the $G$-graded context. Recall that for any $g\in G$, the functor $A_g\otimes_B-:\mathrm{mod}\,B\to \mathrm{mod}\,B$ is an autoequivalence with quasi-inverse $A_{g^{-1}}\otimes_B-$. Therefore, the group $G$ acts on the isomorphism classes of $B$ modules, and we denote ${}^gY=A_g\otimes_BY$.
\end{subsec}

The next theorem is crucial in our study of support $\tau$-tilting modules. The proofs of its two statements are, up to a point, very similar. The second statement is essentially in \cite[Section 5]{art:Dade1984}, but we give here an argument which avoids Dade's theory of Clifford extensions over group graded algebras.

\begin{theorem} \label{t:gr-cover} { Assume that the algebra $B$ is self-injective.}

{\rm1)} Let $Y$ be a  $B$-module. Then $A\otimes_B \mathrm{P}(Y) \simeq \mathrm{P}(A\otimes_B Y)$. Moreover, if $P_1\to P_0\to Y\to 0$ is a minimal projective presentation, then $A\otimes_B P_1\to A\otimes_B P_0\to A\otimes_B Y\to 0$ is a minimal projective presentation.

{\rm2)} Assume that $|G|$ is invertible in $k$, and let $X$ be an $A$-module. Then $\mathrm{Res}^A_B \mathrm{P}(X)\simeq \mathrm{P}(\mathrm{Res}^A_B X)$. Moreover, if $Q_1\to Q_0\to X\to 0$ is a minimal projective presentation, then $\mathrm{Res}^A_B Q_1\to \mathrm{Res}^A_B Q_0\to \mathrm{Res}^A_B X\to 0$ is a minimal projective presentation.
\end{theorem}

\begin{proof} 1) We know that $Y$ is projective if and only $A\otimes_BY$ is projective.

We show that $Y$ is projective-free if and only $A\otimes_BY$ is projective-free. Indeed, assume that $A\otimes_BY$ is projective-free. By the Krull-Schmidt theorem, we may write $Y=Y'\oplus Y''$, where $Y'$ is projective, and $Y''$ is projective-free. We get $A\otimes_B Y=A\otimes_B Y'\oplus A\otimes_B Y''$, where $A\otimes_B Y'$ is projective. But our assumption forces $A\otimes_B Y'=0$, hence $Y'=0$, so $Y$ is projective-free. Conversely, assume that $Y$ is projective-free, and let $A\otimes_B Y=X'\oplus X''$, where $X'$ is projective, and $X''$ is projective-free. We get $\mathrm{Res}^A_B A\otimes_B Y=\mathrm{Res}^A_B X'\oplus \mathrm{Res}^A_B X''$. Note that $\mathrm{Res}^A_B A\otimes_B Y$ is projective-free, because it is the direct sum of the $G$-conjugates of $Y$, and $\mathrm{Res}^A_B X'$ is projective. Our assumption forces $\mathrm{Res}^A_B X'=0$, hence $X'=0$, so $A\otimes_B Y$ is projective-free.

Consider the exact sequence $0\to K\to P_0\to Y\to 0$ in $\mathrm{mod}\,B$, where $P_0=\mathrm{P}(Y)$ is a projective cover of $Y$, and $K=\Omega(Y)$ is projective-free. We obtain the exact sequence $0\to A\otimes_B K\to A\otimes_B P_0\to A\otimes_B Y\to 0$ in $\mathrm{mod}\,A$. Note that $\mathrm{P}(A\otimes_B Y)$ is a direct summand of $A\otimes_B P_0$, but by the above, $A\otimes_B K$ is projective-free, hence $\mathrm{P}(A\otimes_B Y) \simeq A\otimes_B \mathrm{P}(Y)$.

2) It is well known (see \cite{{art:NVV1989}}) that under our  assumptions, the functor $\mathrm{Res}^A_B$ is separable, hence $X$ is projective if and only $\mathrm{Res}^A_B X$ is projective.

We show $X$ is projective-free if and only if $\mathrm{Res}^A_B X$ is projective-free. Indeed, assume that $\mathrm{Res}^A_B X$ is projective-free, and write $X=X'\oplus X''$, where $X'$ is projective, and $X''$ is projective-free. Then $\mathrm{Res}^A_B X=\mathrm{Res}^A_B X'\oplus \mathrm{Res}^A_B X''$, where $\mathrm{Res}^A_B X'$. Our assumption forces $\mathrm{Res}^A_B X'=0$, hence $X'=0$, so $X$ is projective-free. Conversely, assume that $X$ is projective-free, and write $\mathrm{Res}^A_B X=Y'\oplus Y''$, where $Y'$ is projective and $Y''$ is projective-free. We get $\mathrm{Ind}^A_B \mathrm{Res}^A_B X=\mathrm{Ind}^A_B Y'\oplus \mathrm{Ind}^A_B Y''$. But $X$ is a direct summand of $\mathrm{Ind}^A_B \mathrm{Res}^A_B X$, and since $X$ is projective-free, we deduce by 1) that $X$ is a direct summand of $\mathrm{Ind}^A_B Y''$.
{ It follows  that $\mathrm{Res}^A_B X$ is a direct summand of $\mathrm{Res}^A_B \mathrm{Ind}^A_B Y''$. But $\mathrm{Res}^A_B \mathrm{Ind}^A_B Y''$ is a direct sum of $G$-conjugates of $Y''$, hence it is projective-free. Therefore, $\mathrm{Res}^A_B X$ is projective-free.}

Consider the exact sequence $0\to K\to Q_0\to X\to 0$ in $\mathrm{mod}\,A$, where $Q_0=\mathrm{P}(X)$ is a projective cover of $X$, and $K=\Omega(X)$ is projective-free. We obtain the exact sequence $0\to \mathrm{Res}^A_B K\to \mathrm{Res}^A_B Q_0\to \mathrm{Res}^A_B X\to 0$ in $\mathrm{mod}\,B$. Note that $\mathrm{P}(\mathrm{Res}^A_B X)$ is a direct summand of $\mathrm{Res}^A_B Q_0$, but by the above, $\mathrm{Res}^A_B K$ is projective-free, hence $\mathrm{P}(\mathrm{Res}^A_B X) \simeq \mathrm{Res}^A_B \mathrm{P}(X)$.
\end{proof}





\begin{subsec} \label{s:grading-dual} Recall that if $X=\bigoplus X_g$ is a $G$-graded $k$-space, then $\mathrm{D}X=\mathrm{Hom}_k(X, k)$ is also $G$-graded with $g$-component
\[(\mathrm{D}X)_g= \{ f\in \mathrm{D}X \mid  f(X_h)=0 \textrm{ for all } h\in G\setminus\{g^{-1}\}\},\]
and we may identify $(\mathrm{D}X)_g=\mathrm{Hom}_k(X_{g^{-1}},k)$.

If $X$ is an $A$-module, we also denote $X^*=\mathrm{Hom}_A(X,A)$.
\end{subsec}

\begin{proposition} \label{l:dual-bimods-Ind} Let $Y$ be a $B$-module, and denote $X=A\otimes_B Y$. Then

{\rm 1)} $\mathrm{D}(A\otimes_B Y)\simeq \mathrm{D}Y\otimes_B A$ as $G$-graded right $A$-modules.

{\rm 2)} $\mathrm{Hom}_A(A\otimes_B Y,A)\simeq \mathrm{Hom}_B(Y,B)\otimes_B A$ as $G$-graded right $A$-modules.

{\rm 3)} $\nu_A(A\otimes_B Y)\simeq A\otimes_B \nu_B(Y)$ as $G$-graded  $A$-modules.

{\rm 4)} { If $B$ is self-injective, then} $\mathrm{Tr}_A(A\otimes_B Y)\simeq \mathrm{Tr}_BY \otimes_B A$ as $G$-graded right $A$-modules.

{\rm 5)} { If $B$ is self-injective, then} $\tau_A(A\otimes_B Y)\simeq A\otimes_B \tau_B(Y)$ as $G$-graded  $A$-modules.
\end{proposition}

\begin{proof} 1) We have that $\mathrm{D}X$ is a right $A$-module, where $(fa)(x)=f(ax)$ for all $f\in \mathrm{D}X$, $x\in X$ and $a\in A$. It is clear from \ref{s:grading-dual} that $\mathrm{D}X$ is a $G$-graded right $A$-module with $1$-component isomorphic to the right $B$-module $\mathrm{D}Y$. The statement follows by follows by \cite[Theorem 2.8]{art:Dade1980}.

2) We have that $\mathrm{Hom}_A(X,A)$ is a right $A$-module, where $(fa)(x)=f(x)a$ for all $f\in \mathrm{Hom}_A(X,A)$, $x\in X$ and $a\in A$. By \cite[Proposition 3.12]{art:Dade1980}, $\mathrm{Hom}_A(X,A)$ has a $G$-grading, where, via restriction of scalars to $B$,
\[\mathrm{Hom}_A(X,A)_g\simeq \mathrm{Hom}_B(X_h,A_{hg}),\]
for all $h\in G$, where recall that $X_h=A_h\otimes_BY$. It is clear that that $\mathrm{Hom}_A(X,A)$ becomes a $G$-graded right $A$-module, so again the statement follows by follows by \cite[Theorem 2.8]{art:Dade1980}.

3) As $\nu_B=\mathrm{D} \mathrm{Hom}_B(-,B)$, the statement follows immediately by 1) and 2).

4) We may assume that $Y$ is non-projective. Let $P_1\overset{f_1}\longrightarrow P_0\to Y\to 0$ be a minimal projective presentation of $Y$. Then \[P_0^*\to P_1^*\to \mathrm{Tr}_BY\to 0\] is a minimal projective presentation of $\mathrm{Tr}_BY=\mathrm{Coker}\,f_1^*$. By Theorem \ref{t:gr-cover} 1), $A\otimes_B P_1\to A\otimes_B P_0\to A\otimes_B Y\to 0$ is a minimal projective presentation of $A\otimes_B Y$. Then
\[(A\otimes_B P_0)^*\to (A\otimes_B P_1)^*\to \mathrm{Tr}_A(A\otimes_B Y)\to 0\] is a minimal projective presentation of $\mathrm{Tr}_A(A\otimes_B Y)=\mathrm{Coker}\,(A\otimes_B f_1)^*$. Again by Theorem \ref{t:gr-cover} 1),
\[P_0^*\otimes_BA\to P_1^*\otimes_BA\to \mathrm{Tr}_BY\otimes_BA\to 0\] is a minimal projective presentation of $\mathrm{Tr}_BY\otimes_BA$. The statement now follows by 2) and by the fact that $\mathrm{Coker}$ commutes with $\mathrm{Ind}^A_B$.

5) As $\tau=\mathrm{D} \mathrm{Tr}$, the statement  follows immediately by 1) and 4).
\end{proof}

The above results generalize \cite[Lemma 3.20]{art:KK2020}, \cite[Lemma 3.11]{art:KK2021} and \cite[Lemma 3.2]{art:K2022}.

As a consequence of Proposition \ref{l:dual-bimods-Ind} we deduce that all the above functors commute with $G$-conjugation from \ref{s:g-conj}.  The first statement of the next corollary is \cite[Lemma 3.4]{art:MPan2019}, but we give a different proof here.

\begin{corollary} \label{c:conj-nu-au} Let $Y$ be a $B$-module. Then, for any $g\in G$, we have:

{\rm 1)} $\nu ({}^gY)\simeq {}^g(\nu Y)$.

{\rm 2)} {If $B$ is self-injective then} $\tau ({}^gY)\simeq {}^g(\tau Y)$.
\end{corollary}

\begin{proof} 1) By Proposition \ref{l:dual-bimods-Ind} 1), 2) and 3) we see that the $g$-component of $\nu_A(A\otimes_B Y)\simeq A\otimes_B \nu_B(Y)$ is, on the one hand, $\nu_B(A_g\otimes_B Y)$, and it is $A_g\otimes_B \nu_B(Y)$ on the other hand.

2) follows in the same way by Proposition \ref{l:dual-bimods-Ind}  2) and 5).
\end{proof}

We also record the ``$\mathrm{Res}$" version of Proposition \ref{l:dual-bimods-Ind}.

\begin{proposition} \label{l:dual-bimods-Res} Let $X$ be an $A$-module. Then

{\rm 1)} $\mathrm{D}(\mathrm{Res}^A_BX)\simeq \mathrm{Res}^A_B\mathrm{D}X$ as right $B$-modules.

{\rm 2)} $\mathrm{Hom}_B(\mathrm{Res}^A_BX,B)\simeq \mathrm{Res}^A_B\mathrm{Hom}_A(X,A)$ as right $B$-modules.

{\rm 3)} $\nu_B(\mathrm{Res}^A_B X)\simeq \mathrm{Res}^A_B \nu_A(X)$ as  $B$-modules.

{\rm 4)} {If $B$ is self-injective, and} if $|G|$ is invertible in $k$, then $\mathrm{Tr}_B(\mathrm{Res}^A_B X)\simeq \mathrm{Res}^A_B\mathrm{Tr}_A X$ as right $B$-modules.

{\rm 5)}  { If $B$ is self-injective, and} if $|G|$ is invertible in $k$, then $\tau_B(\mathrm{Res}^A_BX)\simeq \mathrm{Res}^A_B \tau_A(X)$ as  $B$-modules.
\end{proposition}

\begin{proof} 1) is obvious (and valid for any extension of $k$-algebras).

2) The adjunction isomorphism from $\mathrm{Hom}_A(X, A\otimes_B B)$ to $\mathrm{Hom}_B(\mathrm{Res}^A_B X, B)$ sends a map $g\in \mathrm{Hom}_A(X, A)$ to the map $g(-)_1\in \mathrm{Hom}_B(\mathrm{Res}^A_B X, B)$ { which sends $x\in X$ to $g(x)_1\in B$}. It is easy to verify that this isomorphism is $B$-linear.

3) As $\nu_A=\mathrm{D} \mathrm{Hom}_A(-,A)$, the statement follows immediately by 1) and 2).

4) We may assume that $Y$ is non-projective. Let $Q_1\overset{g_1}\longrightarrow Q_0\to X\to 0$ be a minimal projective presentation of $X$. Then \[Q_0^*\to Q_1^*\to \mathrm{Tr}_AX\to 0\] is a minimal projective presentation of $\mathrm{Tr}_AX=\mathrm{Coker}\,g_1^*$. By Theorem \ref{t:gr-cover} 2), $\mathrm{Res}^A_B Q_1\to \mathrm{Res}^A_B Q_0\to \mathrm{Res}^A_B X\to 0$ is a minimal projective presentation of $\mathrm{Res}^A_B X$. Then
\[(\mathrm{Res}^A_B Q_0)^*\to (\mathrm{Res}^A_B Q_1)^*\to \mathrm{Tr}_B(\mathrm{Res}^A_B X)\to 0\] is a minimal projective presentation of $\mathrm{Tr}_B(\mathrm{Res}^A_B X)=\mathrm{Coker}\,(\mathrm{Res}^A_B g_1)^*$. Again by Theorem \ref{t:gr-cover} 2),
\[\mathrm{Res}^A_B Q_0^* \to \mathrm{Res}^A_B Q_1^*\to \mathrm{Res}^A_B\mathrm{Tr}_A X\to 0\] is a minimal projective presentation of $\mathrm{Res}^A_B \mathrm{Tr}_A X$. The statement now follows by 2) and by the fact that $\mathrm{Coker}$ commutes with $\mathrm{Res}^A_B $.

5) As $\tau=\mathrm{D} \mathrm{Tr}$, the statement  follows immediately by 1) and 4).
\end{proof}

In the final part of this section we discuss the relation between the $G$-action  on the isomorphism classes of $B$-modules and the $\mathrm{Ind}^A_B$-stability condition from \ref{s:stable}.

\begin{subsec} \label{ss:basicnot}  Let $Y$ be a $B$-module.  Then $\mathrm{End}_A(A\otimes_B Y)^{\textrm{op}}$ is a $G$-graded $k$-algebra with $1$-component naturally isomorphic to $\mathrm{End}_B(Y)^{\textrm{op}}$. Moreover,  $Y$ is a $(B,\mathrm{End}_B(Y)^{\textrm{op}})$-bimodule, and $A\otimes_BY$ is  a $G$-graded $(A,\mathrm{End}_A(A\otimes_B Y)^{\textrm{op}})$-bimodule.

The $B$-module $Y$ is called $G$-invariant if ${}^gY=A_g\otimes_B Y\simeq Y$ as $B$-modules for all $g\in G$, or equivalently, $\mathrm{End}_A(A\otimes_B Y)^{\textrm{op}}$ is a crossed product of $\mathrm{End}_B(M)^{\textrm{op}}$ and $G$.

The $B$-module $Y$ is called weakly $G$-invariant  if ${}^g\mathrm{add}\,Y=\mathrm{add}\,Y$ for all $g\in G$, or equivalently, $\mathrm{End}_A(A\otimes_B Y)^{\textrm{op}}$ is strongly $G$-graded (see \cite[Theorem 6.3]{art:Dade1980}).
\end{subsec}

\begin{subsec} \label{ss:AeA} In particular, $G$ acts on the isomorphism classes of projective summands of $B$. If $e\in B$ is an idempotent, then the projective $B$-module $Be$ is $G$-invariant if and only if $eAe$ is a crossed product of $eBe$ and $G$.

In this case, $Ae$ is a $G$-graded $(A,eAe)$-bimodule. Note that $AeA$ is the image of the multiplication map $Ae\otimes_{eAe}eA\to A$, so it is a $G$-graded subbimodule of ${}_AA_A$. Moreover, the ideal $BeB$ of $B$ is $G$-invariant, that is, $A_gBeB=BeBA_g$ for all $g\in G$, and $AeA=AeB=BeA$ is a $G$-graded ideal of $A$. This immediately implies that $A/AeA$ is also strongly $G$-graded, with $1$ component isomorphic to $B/BeB$.
\end{subsec}

\begin{lemma} \label{stable-G-invar} The $B$-module $Y$ is  $\mathrm{Ind}^A_B$-stable  if and only if $Y$ is weakly $G$-invariant. Moreover, if $Y$ is basic, then  $Y$ is  $\mathrm{Ind}^A_B$-stable if and only if $Y$ is $G$-invariant.

\end{lemma}

\begin{proof} Both statements are obvious, taking into account that $\mathrm{Res}^A_B Y=\bigoplus_{g\in G}A_g\otimes_B Y$, and that $Y$ is weakly $G$-invariant if and only if ${}^gM\in\mathrm{add}\,Y$ for all $g\in G$.
\end{proof}

\begin{subsec} \label{s:Clifford-corr} Finally, we recall from \cite[Section 7]{art:Dade1980} the ``stable Clifford correspondence". Assume that $Y$ is a weakly $G$-invariant $B$-module. Denote by $\mathrm{mod}\textrm{-}(A\vert Y)$ the full subcategory of $\mathrm{mod}\textrm{-}A$ consisting of $A$-modules $X$ such that $\mathrm{Res}^A_BX\in \mathrm{add}\,Y$.

Then the functor $\mathrm{Hom}_A(A\otimes_BY,-)$ induces and equivalence between $\mathrm{mod}\textrm{-}(A\vert Y)$ and $\mathrm{mod}\textrm{-}(\mathrm{End}_A(A\otimes_B Y)^{\textrm{op}}\vert F)$ ($\mathrm{End}_A(A\otimes_B Y)^{\textrm{op}}$-modules which are projective as $F$-modules), with quasi-inverse $(A\otimes_BY)\otimes_{\mathrm{End}_A(A\otimes_BY)^{\textrm{op}}}-$.
Note also that we have isomorphisms of functors \[\mathrm{Hom}_A(A\otimes_BY,-)\simeq \mathrm{Hom}_B(Y,\mathrm{Res}^A_B(-)):\mathrm{mod}\textrm{-}(A\vert Y)\to\mathrm{mod}\textrm{-}(\mathrm{End}_A(A\otimes_B Y)^{\textrm{op}}\vert F)\] and \[(A\otimes_BY)\otimes_{\mathrm{End}_A(A\otimes_BY)^{\textrm{op}}}-\simeq Y\otimes_{F}-:\mathrm{mod}\textrm{-}(\mathrm{End}_A(A\otimes_B Y)^{\textrm{op}}\vert F)
  \to \mathrm{mod}\textrm{-}(A\vert Y)    .\]
\end{subsec}

\section{Inducing and restricting support $\tau$-tilting {pairs for self-injective algebras}}  \label{sec:ind-res-tau}

We continue with the notations and assumptions of the preceding section. Moreover, from now on \textit{we assume that $B$ (equivalently, $A$) is a self-injective algebra}.

\begin{lemma} \label{l:G-inv} Assume that $(M,P)$ is a  support  $\tau$-tilting  pair, where $M$ is basic and $P=Be$. The following statements are equivalent:

$\mathrm{(1)}$ $M$ is a $G$-invariant $B$-module.

$\mathrm{(2)}$ ${}^g\mathrm{Fac}\,M=\mathrm{Fac}\,M$ for all $g\in G$.

$\mathrm{(3)}$ $P$ is a $G$-invariant $B$-module.

$\mathrm{(4)}$  $\mathbf{P}$ is $G$-invariant in the category of complexes of $B$-modules.
\end{lemma}

\begin{proof} The equivalence of the statements is an immediate consequence of \cite[Proposition 2.3 and Theorem 2.7]{art:AIR2014} and Lemma \ref{stable-G-invar}.
\end{proof}

The next theorem is result  the main result of this paper.

\begin{theorem} \label{t:main-ind-2-silt} Assume that $(M,P)$ is a basic support $\tau$-tilting pair of $B$-modules. The following statements are equivalent.

{\rm(1)} $A\otimes_B\mathbf{P}\in 2\textrm{-}\mathrm{silt}\,A$.

{\rm(2)} $(A\otimes_BM, A\otimes_BP)$ is a support $\tau$-tilting pair of $A$-modules.

{\rm(3)} $M$ is a $G$-invariant.

\noindent In this case, if $e$ is an idempotent such that $P=Be$ then $A\otimes_BP$ may me identified with  $Ae$, and $A\otimes_BM$ is a $\tau$-tilting $A/AeA$-module.
\end{theorem}

{ \begin{proof}
For the equivalence of (1) and (3) we will apply Proposition \ref{prop:var-ext-2-silt} for $\mathcal{F}=A\otimes_B-$ and $\mathcal{G}=\mathrm{Res}^A_B-$. In order to do this, we only have to verify the condition (d). By using Proposition \ref{l:dual-bimods-Ind} and Corollary \ref{c:conj-nu-au}, we obtain
\begin{align*} \tau \mathrm{Res}^A_B (A\otimes_BM) & \cong \bigoplus_{g\in G}\tau(A_g\otimes_BM)\cong   \bigoplus_{g\in G} A_g\otimes_B\tau M \\ & \cong \mathrm{Res}^A_B A\otimes_B\tau M\cong  \mathrm{Res}^A_B \tau(A\otimes_B M),\end{align*} and the proof is complete.

To prove that (1)$\Leftrightarrow$(2), we first assume that $A\otimes_B\mathbf{P}\in 2\textrm{-}\mathrm{silt}\,A$, where $A\otimes_B\mathbf{P}=(A\otimes_BP_1\oplus A\otimes_B P\to A\otimes_B P_0)$. We have that $\mathrm{H}^0 (A\otimes_B\mathbf{P}) \simeq A\otimes_B\mathrm{H}^0(\mathbf{P}) \simeq A\otimes_B M$. Moreover, by Theorem \ref{t:gr-cover} 1),  \[A\otimes_BP_1 \to A\otimes_B P_0\to A\otimes_B M\to 0\] is a minimal projective presentation.  By \cite[Theorem 3.2]{art:AIR2014} it follows that  $(A\otimes_BM, A\otimes_BP)$ is a support $\tau$-tilting pair.

Conversely, assume that $(A\otimes_BM, A\otimes_BP)$ is a support $\tau$-tilting pair. Since by Theorem \ref{t:gr-cover} 1), $A\otimes_BP_1 \to A\otimes_B P_0\to A\otimes_B M\to 0$ is a minimal projective presentation, it follows by \cite[Theorem 3.2]{art:AIR2014} that $A\otimes_B\mathbf{P}=(A\otimes_BP_1\oplus A\otimes_B P\to A\otimes_B P_0)$ is a 2-silting complex.

For the last statement, we have seen in \ref{ss:AeA} that $AeA=BeA=AeB$ and that $A/AeA$ is a strongly $G$-graded $k$-algebra. It also follows that $AeA$ annihilates $A\otimes_BM$, and that $M$ is a $G$-invariant $B/BeB$-module. Since $M$ is a $\tau$-tilting $B/BeB$-module, the above arguments imply that $A\otimes_BM\simeq A/AeA\otimes_{B/BeB}M$ is a tilting $A/AeA$-module.
\end{proof}

\begin{remark}
If we are in the hypothesis (3), we can see directly that $|A\otimes_B M|=\mathrm{s}(A\otimes_B M)=|A|-|Ae|$ as follows: Let $Q$ be an indecomposable projective $A$-module which is not a summand of $A\otimes_BP\simeq Ae$. It is enough to show that
$\mathrm{Top}\,Q$ appears as a composition factor of $A\otimes_B M$, that is, $\mathrm{Hom}_A(Q,A\otimes_B M)\neq 0$. Let $R$ be an indecomposable summand of $\mathrm{Res}^A_BQ$. By Clifford theory we have that $\mathrm{Res}^A_BQ$ is a direct sum of $G$-conjugates of $R$, and $Q$ is a summand of $A\otimes_B R$. Clearly, $R$ is not a summand of $P$. Since $(M,P)$ is a support $\tau$-tilting pair, we have that $\mathrm{Hom}_B(R,M)\neq 0$. We deduce that
$\mathrm{Hom}_A(Q,A\otimes_B M) \simeq \mathrm{Hom}_B(\mathrm{Res}^A_BQ,M) \neq 0 .$
\end{remark}
}

We have a ``Res'' version of Theorem  \ref{t:main-ind-2-silt}.

\begin{theorem} \label{t:main-res-2-silt}  Let $(N,Q)$ be a basic support $\tau$-tilting pair of $A$-modules, and let $\mathbf{Q}=(Q_1\oplus Q\overset{[g\, 0]}\longrightarrow Q_0)$ be the corresponding $2$-silting module. If $|G|$ is invertible in $k$, the following statements are equivalent.

{\rm(1)} $\mathrm{Res}^A_B\mathbf{Q}\in 2\textrm{-}\mathrm{silt}\,B$.

{\rm(2)} $(\mathrm{Res}^A_B N, \mathrm{Res}^A_B Q)$ is a support $\tau$-tilting pair of $B$-modules.

{\rm(3)} $N$ is $\mathrm{Res}^A_B$-stable, that is $\mathrm{Ind}^A_B \mathrm{Res}^A_B N\in\mathrm{add}\,N$.
\end{theorem}

\begin{proof}  The equivalence of (1) and (3) again follows { from Proposition \ref{prop:var-ext-2-silt}, for $\mathcal{F}=\mathrm{Res}^A_B-$ and $\mathcal{G}=A\otimes_B-$. In this case, the hypothesis (d) from Proposition \ref{prop:var-ext-2-silt} is satisfied by Proposition \ref{l:dual-bimods-Res} 5). }

Next, we will prove (1)$\Leftrightarrow$(2) Assume that $\mathrm{Res}^A_B \mathbf{Q}\in 2\textrm{-}\mathrm{silt}\,B$, where $\mathrm{Res}^A_B  \mathbf{Q}=(\mathrm{Res}^A_B  Q_1\oplus \mathrm{Res}^A_B Q\to \mathrm{Res}^A_B Q_0)$. We have that $\mathrm{H}^0 (\mathrm{Res}^A_B \mathbf{Q}) \simeq \mathrm{Res}^A_B \mathrm{H}^0(\mathbf{Q}) \simeq \mathrm{Res}^A_B N$. Moreover, by Theorem \ref{t:gr-cover} 2),  \[\mathrm{Res}^A_B Q_1 \to \mathrm{Res}^A_B Q_0\to \mathrm{Res}^A_B N\to 0\] is a minimal projective presentation.  By \cite[Theorem 3.2]{art:AIR2014} it follows that  $(\mathrm{Res}^A_B N, \mathrm{Res}^A_B Q)$ is a support $\tau$-tilting pair.

Conversely, assume that $(\mathrm{Res}^A_B N, \mathrm{Res}^A_B Q)$ is a support $\tau$-tilting pair. Since by Theorem \ref{t:gr-cover} 2), $\mathrm{Res}^A_B Q_1 \to \mathrm{Res}^A_B Q_0\to \mathrm{Res}^A_B N\to 0$ is a minimal projective presentation, it follows by \cite[Theorem 3.2]{art:AIR2014} that $\mathrm{Res}^A_B \mathbf{P}=(\mathrm{Res}^A_B Q_1\oplus \mathrm{Res}^A_B Q\to \mathrm{Res}^A_B Q_0)$ is a 2-silting complex.
\end{proof}

\begin{remark} The $B$-modules $\mathrm{Res}^A_B N$,  $\mathrm{Res}^A_B Q$, and the idempotent $e\in B$ such that $\mathrm{Res}^A_B Q=Be$ can be described using \cite[Theorem 2.4.1]{book:M1999}.
\end{remark}

\begin{corollary} \label{c:tau-1} The bijection between the isomorphism classes of support $\tau$-tilting pairs and the isomorphism classes of support $\tau^{-1}$-tilting pairs commute with $\mathrm{Ind}^A_B$ and $\mathrm{Res}^A_B$.
\end{corollary}

\begin{proof} By \cite[Theorem 2.15]{art:AIR2014}, the bijection is given by $(M,P)\mapsto (\tau\,M\oplus\nu\,P, \nu\,M')$, where $M'$ is the maximal projective summand of $M$. The statement follows by Theorems \ref{t:main-ind-2-silt} and \ref{t:main-res-2-silt}, and the proof of Theorem \ref{t:gr-cover}.
\end{proof}

\section{Semibricks}  \label{s:semibrick}

{ We continue to use the setting of Section \ref{s:tau-tilt-2-silt} and with the assumption that $B$ is self-injective}. { Denote $E=\mathrm{End}_A(A\otimes_BM)^{\textrm{op}}$ and $F=E_1\simeq\mathrm{End}_B(M)^{\textrm{op}}.$}

The following theorem is inspired by and generalizes some results of \cite[Section 4.1]{art:KK2021}, which actually constitutes the primary motivation of our paper.

\begin{theorem} \label{f:bricks} Let $M$ be a $G$-invariant support $\tau$-tilting $B$-module, and let $S=M/L\in \mathrm{f}_\mathrm{L}\textrm{-sbrick}\,B$ be defined as in \ref{ss:notations}. The following statements hold.

{\rm1)} The bijection between the isomorphism classes of left mutations of $M$ and the  isomorphism classes of the bricks $S_1,\dots, S_l$ in $M/\mathrm{rad}_FM$ is $G$-equivariant.

{\rm2)} $\mathrm{Fac}(A\otimes_BM)=\mathrm{T}(A\otimes_BS)$.

{\rm3)} Let $\bar S\in \mathrm{f}_\mathrm{L}\textrm{-}\mathrm{sbrick}\ A$ be the left-finite semibrick corresponding to $A\otimes_BM$. Then $\mathrm{Res}^A_B\bar S$ belongs to $\mathrm{Fac}\,S$.

{\rm4)} There is a bijection between the isomorphism classes of left finite $A$-semibricks belonging to $\mathrm{mod}\,(A|S)$ and the isomorphism classes of left finite $\mathrm{End}_A(A\otimes_BS)^{\mathrm{op}}$-semibricks.

{\rm5)} Let $S_i\in \mathrm{f}_\mathrm{L}\textrm{-}\mathrm{brick}\,B$ be a direct summand of $S$. Assume that the $B$-module structure of $S_i$ extends to an $A$-module structure.  Then there is a bijection between the isomorphism classes of extensions of $S_i$ to $A$ and the isomorphism classes of left finite $\mathrm{End}_A(A\otimes_BS_i)^{\mathrm{op}}$-semibricks of $\mathrm{End}_B(S_i)^{\mathrm{op}}$-dimension $1$.

{\rm6)} Assume that the nonisomorphic direct summands $S_i$ and $S_j$ of $S$ extend to the $A$-modules $\tilde S_i$ and $\tilde S_j$, respectively. Then $\mathrm{Hom}_A(\tilde S_i,\tilde S_j) =0$.
\end{theorem}

\begin{proof} 1) Since $M$ is $G$-invariant, the set of isomorphism classes of the bricks in $M/\mathrm{rad}_FM$ is $G$-invariant, and by  Corollary \ref{c:conj-nu-au} and \cite[Proposition 2.12]{art:Asai2020}, the support $\tau^{-1}$-tilting $B$-module $N$ corresponding to $M$ is also $G$-invariant. Let $M'< M$ be the left mutation of $M$ corresponding to $S_i$ (see \cite[Definition 2.14]{art:Asai2020}), and let $N'$ be the support $\tau^{-1}$-tilting $B$-module corresponding to $M'$, { so $N'$ is a left mutation of $N$, by \cite[Theorem 2.15]{art:Asai2020}. For $g\in G$, $g$-conjugation is a Morita autoequivalence of $\mathrm{mod}\,B$, hence  ${}^gM'$ is a left mutation of ${}^gM\simeq M$, and ${}^gN'$ is a left mutation of ${}^gN\simeq N$. }
By  \cite[Proposition 2.17]{art:Asai2020}, $S_i$ is the unique brick in $\mathrm{Fac}\,M\cap \mathrm{Sub}\,N'$.  Then ${}^gS_i$ is the unique brick in ${}^g(\mathrm{Fac}\,M\cap \mathrm{Sub}\,N')=\mathrm{Fac}\,M\cap \mathrm{Sub}\,{}^gN'$. Consequently, the brick ${}^gS_i$  corresponds to the left mutation ${}^gM'$ of $M$.

2) The epimorphism $M\to S$ of $B$-modules induces the epimorphism $A\otimes_BM\to A\otimes_BS$ of $A$-modules. Since $A\otimes_BM$ is support $\tau$-tilting, $\mathrm{Fac}\,A\otimes_BM$ is a torsion class, hence $\mathrm{T}(A\otimes_BS)\subseteq \mathrm{Fac}(A\otimes_BM)$. Conversely, since $M\in \mathrm{T}(S)=\mathrm{Filt}(\mathrm{Fac}(S))$, it follows easily by the exactness of $A\otimes_B-$ that $A\otimes_BM\in \mathrm{Filt}(\mathrm{Fac}(A\otimes_BS))=\mathrm{T}(A\otimes_BS)$.

3) { Since $\mathrm{J}_{\mathrm{gr}}(E)=\mathrm{J}(F)E=E\mathrm{J}(F)\subseteq \mathrm{J}(E)$,} we have the $A$-module epimorphisms
\[A\otimes_BM\to A\otimes_BM/(A\otimes_BM)\mathrm{J}_{\mathrm{gr}}(E) \to A\otimes_BM /(A\otimes_BM)\mathrm{J}(E) =\bar S. \]
We restrict this to $B$, and notice that $A\otimes_BM$ is a direct sum of copies of $M$, while
\[A\otimes_BM/(A\otimes_BM)\mathrm{J}_{\mathrm{gr}}(E) \simeq \mathrm{Res}^A_B A\otimes_B M/M \mathrm{J}(F)\] is a direct sum of copies of $S$, hence $\mathrm{Res}^A_B\bar S\in \mathrm{Fac}\,S$.

4) By the stable Clifford correspondence \cite[Theorem 7.4]{art:Dade1980} (see \ref{ss:basicnot}  and \ref{s:Clifford-corr}), the functors
\[\mathrm{Hom}_A(A\otimes_BS,-)\simeq \mathrm{Hom}_B(S,\mathrm{Res}^A_B(-))\] and
\[(A\otimes_BS)\otimes_{\mathrm{End}_A(A\otimes_BS)^{\mathrm{op}}}- \simeq S\otimes_{\mathrm{End}_B(S)^{\mathrm{op}}}\mathrm{Res}^{\mathrm{End}_A(A\otimes_BS)^{\mathrm{op}}}_{\mathrm{End}_B(S)}(-) \]
induce an equivalence between $\mathrm{mod}\,(A|S)$ and $\mathrm{mod}\,(\mathrm{End}_A(A\otimes_BS)^{\mathrm{op}} | \mathrm{End}_B(S)^{\mathrm{op}})$. But the algebra $\mathrm{End}_B(S)^{\textrm{op}}$ is semisimple, hence the latter category coincides with  $\mathrm{mod}\,\mathrm{End}_A(A\otimes_BS)^{\mathrm{op}}$.

5) Let $\tilde S_i$ be an $A$-module whose restriction to $B$ is isomorphic to $S$. Then by Miyashita's theorem \cite[Theorem 1.4.8]{book:M1999}, the group $G$ acts on the algebra $\mathrm{End}_B(S_i)$, and $\mathrm{End}_A(\tilde S_i)$ is isomorphic to the $G$-fixed subalgebra $\mathrm{End}_B(S_i)^G$ of $\mathrm{End}_B(S_i)$. Therefore, $\tilde S_i$ is a brick, and it obviously belongs to $\mathrm{mod}\,(A|S_i)$.

As in 4), the category $\mathrm{mod}\,(A|S_i)$ is equivalent to $\mathrm{mod}\,\mathrm{End}_A(A\otimes_BS_i)^{\mathrm{op}}$. Moreover, we have that \[\dim_{\mathrm{End}_B(S_i)^{\mathrm{op}}} \mathrm{Hom}_B(S_i,\tilde S_i)=1.\]

6) Again by \cite[Theorem 1.4.8]{book:M1999}, $\mathrm{Hom}_B(S_i, S_j)$ is a $kG$-module, and $\mathrm{Hom}_A(\tilde S_i, \tilde S_j)=\mathrm{Hom}_B(S_i, S_j)^G$. Since $\mathrm{Hom}_B(S_i, S_j)=0$, the statement follows.
\end{proof}

\begin{corollary} \label{c:sep-res} 
If $|G|$ is invertible in $k$, then:

{\rm a)} $\mathrm{mod}\,(A|S)= \mathrm{add}\,(A\otimes_BS) $.

{\rm b)}  $A\otimes_BS$ is an $A$-semibrick.
\end{corollary}

\begin{proof} a) Obviously, $\mathrm{add}\,(A\otimes_BS)\subseteq \mathrm{mod}\,(A|S)$. Conversely, let $U\in\mathrm{mod}\,(A|S)$. Then $\mathrm{Res}^A_BU$ is a direct summand of $\mathrm{Res}^A_B (A\otimes_BU)$, so by our assumption, $U$ is a direct summand of $ A\otimes_BU$. Consequently, $U\in \mathrm{add}\,(A\otimes_BS)$.

b) We have that $\mathrm{End}_A(A\otimes_BS)^{\mathrm{op}}$ is a crossed product of $\mathrm{End}_B(S)^{\mathrm{op}}$ and $G$. But $\mathrm{End}_B(S)^{\mathrm{op}}$ is semisimple, so by our assumption and a graded version of Maschke's theorem, $\mathrm{End}_A(A\otimes_BS)^{\mathrm{op}}$ is also semisimple. By Theorem \ref{f:bricks}. 4), the direct summands of $A\otimes_BS$ correspond to the simple submodules of $\mathrm{End}_A(A\otimes_BS)^{\mathrm{op}}$, hence $A\otimes_BS$ is a semibrick.
\end{proof}

\begin{remark} Denote by $\mathcal{T}$, as in \ref{s:bijections}, the functorially finite torsion class determined by $M$, and by $\bar{\mathcal{T}}$ the functorially finite torsion class determined by $A\otimes_BM$. We know that $\mathcal{T}=\mathrm{Fac}\,M=\mathrm{T}(S)$, and $\bar{\mathcal{T}}=\mathrm{Fac}\,A\otimes_B M=\mathrm{T}(\bar S)$, where $\bar S\in \mathrm{f}_\mathrm{L}\textrm{-}\mathrm{sbrick}\ A$ be the left-finite semibrick corresponding to $A\otimes_BM$. From Theorem \ref{f:bricks} 2) and 3) (or the first remark in the proof of Proposition \ref{prop:var-ext-2-silt}) it is easy to deduce that
\[\bar{\mathcal{T}}=\{ X\in\mathrm{mod}\,A \mid \mathrm{Res}^A_BX\in \mathcal{T}  \}.\]
In the same vein, we obtain that if $|G|$ is invertible in $k$, then \[ \mathcal{T}=\{ Y\in\mathrm{mod}\,B \mid \mathrm{Ind}^A_BY\in \bar{\mathcal{T}} \} \]
\end{remark}

\begin{remark} The main results  of \cite{art:KK2020}, \cite{art:KK2021} and \cite{art:K2022} are obtained under the assumption that $k$ is algebraically closed and that any left finite $B$-brick (sometimes any indecomposable $B$-module) is $G$-invariant. With the notation of \ref{f:bricks}, these assumptions imply that $\mathrm{End}_B(S_i)=k$, and  $\mathrm{End}_A(A\otimes_BS_i)^{\mathrm{op}}$ is a twisted group algebra $k_\alpha G$ for some $2$-cocycle $\alpha\in \mathrm{Z}^2(G,k^*)$.

In \cite{art:KK2020} and \cite{art:K2022} it is assumed that $G$ is a $p$-group. This implies that $\mathrm{H}^2(G,k^*)=1$, and that $A\otimes_BY$ is an indecomposable $A$-module whenever $Y$ is an indecomposable $B$-module.

In \cite{art:KK2021} the following assumptions are made:
\begin{itemize}
\item $\mathrm{H}^2(G,k^*)=1$. This implies that any left finite $B$-brick $S_i$ extends to an $A$-brick $\tilde S_i$.
\item  $kG$ is basic. This implies that $A\otimes_B S_i$ is basic.
\end{itemize}
Their results are applied to the case when $B$ is a block with cyclic defect groups.
\end{remark}

\begin{remark} By using Corollary \ref{c:tau-1}, one may obtain dual versions of the statements of Theorem \ref{f:bricks}, involving right finite semibricks, as in\cite{art:KK2021}.
\end{remark}

\section{Wide subcategories as module categories} \label{sec:wide}

We contiunue by using the notations $(M,P)$ and $\mathbf{P}$ as in Section \ref{s:tau-tilt-2-silt}, and the assumption that$B$ is self-injective.

Asai showed in \cite[Theorems 2.27 and 3.15]{art:Asai2020} that certain wide subcategories associated to $\mathbf{P}$ are equivalent to module categories. In this section we study the behaviour of those equivalences with respect to induction and restriction.

\begin{subsec} We will use the following notations.

a) Let \[E=\mathrm{End}_A(A\otimes_BM)^{\textrm{op}}, \qquad F=E_1\simeq\mathrm{End}_B(M)^{\textrm{op}}.\]
Denote by $f_i\in F$ the idempotent endomorphism $M\to M_i\to M$, and let \[f=\sum_{i= l+1}^mf_i.\]

b) Denote $\mathcal{W}=\mathrm{Filt}\,S$, which is the left finite wide subcategory of $\mathrm{mod}\,B$ determined by $S$. Consider the full subcategory
\[\bar{\mathcal{W}}:=\{ U\in \mathrm{mod}\,A \mid \mathrm{Res}^A_BU\in \mathcal{W} \}\]
of $\mathrm{mod}\,A$. Note that by Corollary \ref{c:sep-res},  if $|G|$ is invertible in $k$, then $\bar{\mathcal{W}}=\mathrm{Filt}\,A\otimes_BS$.
\end{subsec}

\begin{theorem} \label{t:wide-ind} Assume that B is self-injective. Let $M$ be a $G$-invariant support $\tau$-tilting $B$-module. The following statements hold.

{\rm1)} The ideal $FfF$ of $F$ is $G$-invariant, and $\mathrm{mod}\,E/EfE$ is a crossed product of $F/FfF$ and $G$.

{\rm2)} $\bar{\mathcal{W}}$ is a  wide subcategory of $\mathrm{mod}\,A$.

{\rm3)} The equivalence  \[\mathrm{Hom}_A(A\otimes_BM,-):\mathrm{Fac}\,A\otimes_BM \to \mathrm{Sub}_{E}\mathrm{D}(A\otimes_BM)\]
restricts to an equivalence between $\bar{\mathcal{W}}$ and $\mathrm{mod}\,E/EfE$, such that the following diagram is commutative:
\[\xymatrix{
	\bar{\mathcal{W}} \ar@<2pt>[rrr]^{\mathrm{Hom}_A(A\otimes_BM,-)} \ar@<2pt>[dd]^{\mathrm{Res}^A_B-} & &
 & \mathrm{mod}\,E/EfE \ar@<2pt>[lll]^{(A\otimes_BM)\otimes_E-} \ar@<2pt>[dd]^{\mathrm{Res}^E_F-} \\ \\
	\mathcal{W}\ar@<2pt>[rrr]^{\mathrm{Hom}_B(M,-)} \ar@<2pt>[uu]^{A\otimes_B-} & & & \mathrm{mod}\,F/FfF \ar@<2pt>[lll]^{M\otimes_F-} \ar@<2pt>[uu]^{E\otimes_F-}
}
\]
\end{theorem}

\begin{proof} 1) Follows immediately by \ref{ss:indices}, the definition of $f$, and by \ref{ss:AeA}.

2) Since $\mathcal{W}:=\mathrm{Filt}\,S$ is a wide subcategory of $\mathrm{mod}\,B$, and the functor $\mathrm{Res}^A_B$ is exact, it is easy to see that $\bar{\mathcal{W}}$ is a  wide subcategory of $\mathrm{mod}\,A$.

3) Since $A\otimes_BM$ is a $\tau$-tilting $A$-module, the functor $\mathrm{Hom}_A(A\otimes_BM,-):\mathrm{Fac}\,A\otimes_BM \to \mathrm{Sub}_{E}\mathrm{D}(A\otimes_BM)$ is an equivalence with inverse $(A\otimes_BM)\otimes_E-$, by \cite[Proposition 3.2]{art:DIY2019}.

The $G$-invariance of $M$ and the exactness of the induction and the restriction functors imply that $A\otimes_B-$ and $\mathrm{Res}^A_B$ restrict to functors between $\mathrm{Fac}\,M$ and $\mathrm{Fac}\,A\otimes_BM$, and also to functors between $\mathcal{W}$ and $\bar{\mathcal{W}}$.

By \ref{ss:AeA} we have that $EfE$ is the $G$-graded ideal of $E$ generated by $FfF$, so it is clear that $E\otimes_F-$ and $\mathrm{Res}^E_F$ restrict to functors between $\mathrm{mod}\,F/FfF$ and $\mathrm{mod}\,E/EfE$.

The $G$-invariance of $M$ implies that $A\otimes_BM\simeq M\otimes_FE$ as $G$-graded $(A,E)$-bimodules. From here and from \ref{l:dual-bimods-Ind}   we immediately get that the functors $E\otimes_F-$ and $\mathrm{Res}^E_F$ restrict to functors between $\mathrm{Sub}_{D}\mathrm{D}(M)$ and $\mathrm{Sub}_{E}\mathrm{D}(A\otimes_BM)$; moreover the commutativity of the diagram also follows.

By \cite[Theorem 2.27]{art:Asai2020}, the bottom row of the diagram is an equivalence. If $X\in\bar{\mathcal{W}}$, then $\mathrm{Hom}_A(A\otimes_BM,X)$ is an $E$-module isomorphic to $\mathrm{Hom}_B(M,\mathrm{Res}^A_B X)$. But $\mathrm{Res}^A_B X\in\mathcal{W}$, so by \cite[Theorem 2.27]{art:Asai2020}, $\mathrm{Hom}_A(A\otimes_BM,U)$ is annihilated by $FfF$, hence it belongs to $\mathrm{mod}\,E/EfE$. Conversely, if $U\in\mathrm{mod}\,E/EfE$, then $(A\otimes_BM)\otimes_E U$ is an $A$-module isomorphic to $(M\otimes_FE)\otimes_E U\simeq M\otimes_F U$, so again by \cite[Theorem 2.27]{art:Asai2020}, its restriction to $B$ belongs to $\mathcal{W}$. Consequently, the top row is also an equivalence.
\end{proof}

\begin{subsec} We introduce another set of notations. Let
\[C=\mathrm{End}_{\mathcal{K}^{\mathrm{b}}(\mathrm{proj}\,A)}(A\otimes_B\mathbf{P})^{\textrm{op}}, \qquad D=C_1\simeq\mathrm{End}_{\mathcal{K}^{\mathrm{b}}(\mathrm{proj}\,B)}(\mathbf{P})^{\textrm{op}},\]
and
 \[C'=\mathrm{End}_{\mathcal{K}^{\mathrm{b}}(\mathrm{inj}\,A)}(A\otimes_B\nu\mathbf{P})^{\textrm{op}}, \qquad D'=C'_1\simeq\mathrm{End}_{\mathcal{K}^{\mathrm{b}}(\mathrm{inj}\,B)}(\nu\mathbf{P})^{\textrm{op}},\]
Let $d_i\in D$ be the idempotent endomorphism $\mathbf{P}\to P_i\to \mathbf{P}$, and let \[d=\sum_{i=l+1}^n d_i.\] Similarly, let $d'_i\in D'$ be the idempotent endomorphism $\nu \mathbf{P}\to \nu P_i\to \nu \mathbf{P}$, and let \[d'=\sum_{i=l}^l d'_i.\]
\end{subsec}

\begin{subsec} Recall that the heart of the intermediate t-structure corresponding to $\mathbf{P}$ is
\[\mathcal{H}=\mathbf{P}[\neq0]^\perp =\{X\in \mathcal{D}^{\mathrm{b}}(B)   \mid \mathrm{Hom}_{\mathcal{D}^{\mathrm{b}}(\mathrm{mod}\,B)} (\mathbf{P}[n],X)=0  \textrm{ for all } n\neq 0   \}.\]
Recall also that $B$-bricks are identified with simple objects of $\mathcal{H}$.

Let  $\bar{\mathcal{H}}$ denote the heart of the t-structure corresponding to $A\otimes_B\mathbf{P}$. Note that because of the adjunction between induction and restriction, we immediately obtain that
\[\bar{\mathcal{H}} = \{ Y\in \mathcal{D}^{\mathrm{b}}(A)    \mid \mathrm{Res}^A_BY\in \mathcal{H} \}.  \]
\end{subsec}

\begin{theorem} \label{t:heart} Assume that B is self-injective. Let $M$ be a $G$-invariant support $\tau$-tilting $B$-module. The following statements hold.

{\rm1)} The ideal $DdD$ of $D$ is $G$-invariant, and $\mathrm{mod}\,C/CdC$ is a crossed product of $D/DdD$ and $G$.

{\rm2)} The equivalence  \[\mathrm{Hom}_{\mathcal{D}^{\mathrm{b}}(\mathrm{mod}\,A)}(A\otimes_B\mathbf{P},-):\bar{\mathcal{H}} \to \mathrm{mod}\,C\]
restricts to an equivalence between $\bar{\mathcal{W}}$ and $\mathrm{mod}\,C/CdC$, such that the following diagram is commutative:
\[\xymatrix{
	\bar{\mathcal{W}} \ar@<2pt>[rrrr]^{\mathrm{Hom}_{\mathcal{D}^{\mathrm{b}}(\mathrm{mod}\,A)}(A\otimes_B\mathbf{P},-)} \ar@<2pt>[dd]^{\mathrm{Res}^A_B-} & & &  & \mathrm{mod}\,C/CdC  \ar@<2pt>[dd]^{\mathrm{Res}^C_D-} \\ \\
	\mathcal{W}\ar@<2pt>[rrrr]^{\mathrm{Hom}_{\mathcal{D}^{\mathrm{b}}(\mathrm{mod}\,B)}(\mathbf{P},-)} \ar@<2pt>[uu]^{A\otimes_B-} & & & & \mathrm{mod}\,D/DdD  \ar@<2pt>[uu]^{C\otimes_D-}
}
\]
\end{theorem}

\begin{proof}  1) Follows immediately by \ref{ss:indices}, the definition of $d$, and by \ref{ss:AeA}.

2) It is clear that $A\otimes_B-$ and $\mathrm{Res}^A_B$ induce functors between $\mathcal{H}$ and $\bar{\mathcal{H}}$, and restrict to functors between $\mathcal{W}$ and $\bar{\mathcal{W}}$. Theorem \ref{t:wide-ind} shows that $\bar{\mathcal{W}}$ is indeed the left finite wide subcategory corresponding to $A\otimes_BM$, and hence to $A\otimes_B\mathbf{P}$. Therefore, both horizontal equivalences exist by \cite[Theorem 3.15]{art:Asai2020}.

By adjunction, we clearly have
\[\mathrm{Res}^C_D\mathrm{Hom}_{\mathcal{D}^{\mathrm{b}}(\mathrm{mod}\,A)}(A\otimes_B\mathbf{P},\mathbf{X}) \simeq \mathrm{Hom}_{\mathcal{D}^{\mathrm{b}}(\mathrm{mod}\,B)} (\mathrm{Res}^A_B\mathbf{P},\mathbf{X}) \]
for any $\mathbf{X}\in \mathcal{D}^{\mathrm{b}}(\mathrm{mod}\,A)$. On the other hand, if $\mathbf{Y}\in \mathcal{D}^{\mathrm{b}}(\mathrm{mod}\,B)$, then $\mathrm{Hom}_{\mathcal{D}^{\mathrm{b}}(\mathrm{mod}\,A)}(A\otimes_B\mathbf{P},A\otimes_B \mathbf{Y})$ is a $G$-graded $C$-module with $1$-component $\mathrm{Hom}_{\mathcal{D}^{\mathrm{b}}(\mathrm{mod}\,B)}(\mathbf{P},\mathbf{Y})$, hence
\[\mathrm{Hom}_{\mathcal{D}^{\mathrm{b}}(\mathrm{mod}\,A)}(A\otimes_B\mathbf{P},A\otimes_B \mathbf{Y}) \simeq  C\otimes_D\mathrm{Hom}_{\mathcal{D}^{\mathrm{b}}(\mathrm{mod}\,B)}(\mathbf{P},\mathbf{Y}).\]
This shows that the diagram is commutative, and concludes the proof of the theorem.
\end{proof}

\begin{subsec}The next result is the dual version of Theorem \ref{t:heart}, and the proof is similar, by using Proposition \ref{l:dual-bimods-Ind} 5) and Corollary \ref{c:tau-1}.  Denote $\mathcal{W}'=\mathrm{Filt}\,S'$, which is the right finite wide subcategory of $\mathrm{mod}\,B$ determined by $S$. Consider the full subcategory
\[\bar{\mathcal{W}'}:=\{X \in \mathrm{mod}\,A \mid \mathrm{Res}^A_B X\in \mathcal{W}'\} \]
of $\mathrm{mod}\,A$.
\end{subsec}

\begin{theorem} \label{t:heart-dual} Assume that B is self-injective. Let $M$ be a $G$-invariant support $\tau$-tilting $B$-module. The following statements hold.

{\rm1)} The ideal $D'd'D'$ of $D'$ is $G$-invariant, and $\mathrm{mod}\,C'/C'd'C'$ is a crossed product of $D'/D'd'D'$ and $G$.

{\rm2)} The equivalence  \[\mathrm{DHom}_{\mathcal{D}^{\mathrm{b}}(\mathrm{mod}\,A)}(-,A\otimes_B\nu \mathbf{P}):\bar{\mathcal{H}} \to \mathrm{mod}\,C'\]
restricts to an equivalence between $\bar{\mathcal{W}}'[1]$ and $\mathrm{mod}\,C'/C'd'C'$.
\end{theorem}

\subsection*{Acknowledgement} This research is supported by a grant of the Ministry of Research, Innovation and Digitization, CNCS/CCCDI--UEFISCDI, project number PN-III-P4-ID-PCE-2020-0454, within PNCDI III.

\medskip
The authors are grateful to Ryotaro Koshio and  Yuta Kozakai, whose observations helped to significantly improve the manuscript.



\end{document}